\documentclass[a4paper,12pt,twoside]{amsart}
\usepackage[ansinew]{inputenc}
\usepackage[T1]{fontenc}

\usepackage{enumitem}  
\usepackage{a4wide}

\usepackage{mathtools,amssymb,amsthm,amscd}

\usepackage{lmodern}

\usepackage{ifpdf}
\ifpdf\usepackage[pdftex]{hyperref}
\else\usepackage[hypertex]{hyperref}\fi

\hypersetup{pdfpagemode=FullScreen,colorlinks=true}

\newtheorem{thm}{Theorem}[section]
\newtheorem{pro}[thm]{Proposition}
\newtheorem{dfi}[thm]{Definition}

\newtheorem{cor}[thm]{Corollary}
\newtheorem{rem}[thm]{Remark}
\newtheorem{lem}[thm]{Lemma}

\newcommand{\R}{\mathbb{R}}
\newcommand{\C}{\mathbb{C}}
\newcommand{\N}{\mathbb{N}}

\newcommand{\im}{\mathop{\mathrm{Im}}}

\def\eps{\epsilon}

\def\S{\mathcal{S}}
\def\F{\mathcal{F}}
\def\Fi{\mathcal{F}^{-1}}

\begin{document}

\title[On the $\ell^{q,p}$ cohomology of Carnot groups]{On the
  $\ell^{q,p}$ cohomology of Carnot groups}

\author{Pierre Pansu and Michel Rumin} \address{\noindent Laboratoire
  de Math\'ematiques d'Orsay \\\noindent Universit\'e Paris-Sud, CNRS
  \\\noindent Universit\'e Paris-Saclay, 91405 Orsay, France.}

\email{pierre.pansu@math.u-psud.fr and michel.rumin@math.u-psud.fr}

\date{\today}

\begin{abstract}
  We study the simplicial $\ell^{q,p}$ cohomology of Carnot groups
  $G$. We show vanishing and non-vanishing results depending of the
  range of the $(p,q)$ gap with respect to the weight gaps in the Lie
  algebra cohomology of $G$.
\end{abstract}



\thanks{Both authors supported in part by MAnET Marie Curie Initial
  Training Network. P.P. supported by Agence Nationale de la
  Recherche, ANR-15-CE40-0018.}

\maketitle


\section{Introduction}

\subsection{$\ell^{q,p}$ cohomology}

Let $T$ be a countable simplicial complex. Given
$1\leq p\leq q\leq +\infty$, the $\ell^{q,p}$ cohomology of $T$ is the
quotient of the space of $\ell^p$ simplicial cocycles by the image of
$\ell^q$ simplicial cochains by the coboundary $d$,
\begin{eqnarray*}
  \ell^{q,p}H^k(T)=(\ell^p C^k(T)\cap \ker d)/d(\ell^q
  C^{k-1}(T)) \cap \ell^p C^k(T) . 
\end{eqnarray*}
It is a quasiisometry invariant of bounded geometry simplicial
complexes whose usual cohomology vanishes in a uniform manner, see
\cite{Bourdon-Pajot,Elek,Genton,Gromov,Pansu2}. Riemannian manifolds
$M$ with bounded geometry admit quasiisometric simplicial complexes (a
construction is provided below, in Section \ref{sec:leray}). Uniform
vanishing of cohomology passes through. Therefore one can take the
$\ell^{q,p}$ cohomology of any such complex as a definition of the
$\ell^{q,p}$ cohomology of $M$.

One should think of $\ell^{q,p}$ cohomology as a (large scale)
topological invariant. It has been useful in several contexts, mainly
for the class of hyperbolic groups where the relevant value of $q$ is
$q=p$, see \cite{Bourdon-Kleiner,Bourdon-Pajot,
  Cornulier-Tessera,Drutu-Mackay} for instance. It is interesting to
study a class of spaces where values of $q\not=p$ play a significant
role. The goal of the present paper is to compute $\ell^{q,p}$
cohomology, to some extent, for certain Carnot groups. Even the case
of abelian groups is not straightforward.

\subsection{Carnot groups}

Let $G$ be a Carnot group, i.e. a simply connected real Lie group
whose Lie algebra $\mathfrak{g}$ is equipped with a derivation whose
invariant vectors generate $\mathfrak{g}$. The derivation defines
gradations, called weight, on $\mathfrak{g}$ and
$\Lambda^\cdot\mathfrak{g}^*$. The cohomology of $\mathfrak{g}$ is
graded by degree and weight,
\begin{eqnarray*}
  H^\cdot(\mathfrak{g})=\bigoplus_{k,w}H^{k,w}(\mathfrak{g}).
\end{eqnarray*}
For $k=0,\ldots,\mathrm{dim}(\mathfrak{g})$, let $w_{min}(k)$
(resp. $w_{max}(k)$) be the smallest (resp. the largest) weight $w$
such that $H^{k,w}(\mathfrak{g})\not=0$.

\subsection{Main result}

\begin{thm}\label{1}
  Let $G$ be a Carnot group of dimension $n$ and of homogeneous
  dimension $Q$. Let $k=1,\ldots,n$. Denote by
  \begin{eqnarray*}
    \delta N_{max}(k)=w_{max}(k)-w_{min}(k-1),\quad \delta
    N_{min}(k)=\max\{1,w_{min}(k)-w_{max}(k-1)\}. 
  \end{eqnarray*}
  Let $p$ and $q$ be real numbers.
  \begin{enumerate}[label=(\roman*)]
  \item If
    \begin{eqnarray*}
      1<p,q<\infty \quad \text{and}\quad
      \frac{1}{p}-\frac{1}{q}\geq\frac{\delta N_{max}(k)}{Q}.  
    \end{eqnarray*} 
    Then the $\ell^{q,p}$ cohomology in degree $k$ of $G$ vanishes.
  \item If
    \begin{eqnarray*}
      1\leq p,q\leq\infty \quad \text{and}\quad
      \frac{1}{p}-\frac{1}{q}<\frac{\delta N_{min}(k)}{Q}, 
    \end{eqnarray*} 
    then the $\ell^{q,p}$ cohomology in degree $k$ of $G$ does not
    vanish.
  \end{enumerate}
\end{thm}

The non-vanishing statement has a wider scope, see Theorem
\ref{thm:nonzero}. It holds in particular on more general homogeneous
groups.

Theorem \ref{1} is sharp when both $H^{k-1}(\mathfrak{g})$ and
$H^{k}(\mathfrak{g})$ are concentrated in a single weight. This
happens in all degrees for abelian groups and for Heisenberg groups,
for instance. This happens for all Carnot groups in degrees $1$ and
$n$:
$\delta N_{max}(1)=\delta N_{min}(1)=1=\delta N_{max}(n)=\delta
N_{min}(n)$.

Even when not sharp, the result seems of some value as it relates
large scale quasi-isometric analytic invariants of $G$ to its
infinitesimal Lie structure. For instance for $k=2$, the weights on
$H^2(\mathfrak{g})$ can be interpreted as the depth of the relations
defining $G$ with respect to a free Lie group over $\mathfrak{g}_1$,
see e.g. \cite{Rumin_TSG}. The results yield global discrete
Poincar\'e inequalities of type
$\|d^{-1}\omega \|_q \leq C \|\omega\|_p$ on $2$-cocycles, as long as
$1<p, q< +\infty$ and $\frac1p - \frac1q \geq \frac{w_{max}(2)-1}{Q}$,
while there exist $\ell^p$ $2$-cocycles without $\ell^q$ primitive
when $\frac1p - \frac1q < \frac{w_{min}(2)-1}{Q}$.

\smallskip

We shall also illustrate the results on the Engel group in
Section~\ref{sec:engels-group}, and show in particular that, apart
from degrees $1$ and $n$, the natural Carnot homogeneous structure
does not give the best range for the non-vanishing result in general.

\subsection{Method}

We briefly describe the scheme of the proof of Theorem~\ref{1}. The
first step of the vanishing statement is a Leray type lemma which
relates the discrete $\ell^{q,p}$ cohomology to some Sobolev $L^{q,p}$
cohomology of differential forms. This is proved here in the more
general setting of manifolds of bounded geometry of some high
order. One has to take care of an eventual lack of uniformity in the
coverings in order to be able to use local Poincaré inequalities.

A feature of this Sobolev $L^{q,p}$ cohomology is that its forms are
\emph{a priori} quite smooth, and need only be integrated into much
less regular ones. This is because $\ell^p$-cochains at the discrete
level transform into smooth forms made from a smooth partition of
unity, while reversely, less regular $L^p$ forms can still be
discretized into some $\ell^p$ data.

We then focus on Carnot groups. Although they possess dilatations, the
de Rham differential is not homogeneous when seen as a differential
operator.  Nevertheless, as observed in \cite{Rumin_CRAS_1999}, it has
some type of graded hypoellipticity, that can be used to produce
global homotopies $K$. These homotopies are pseudodifferential
operators as studied by Folland and Christ-Geller-G{\l}owacki-Polin in
\cite{CGGP,Folland_1975}. They can be thought of as a kind of
generalized Riesz potentials, like $\Delta^{-1}\delta$ on $1$-forms,
but adapted here to the Carnot homogeneity of the group.

One needs then to translate the graded Sobolev regularity of $K$ into
a standard one to get the $L^{q,p}$ Sobolev controls on $d$. This is
here that the weight gaps of forms arise to control the $(p,q)$
range. Actually, in order to reduce the gap to cohomology weights in
$H^*(\mathfrak{g})$ only, we work with a \emph{contracted de Rham
  complex} $d_c$ instead, available in Carnot geometry. It shares the
same graded analytic regularity as $d$, but uses forms with retracted
components over $H^*(\mathfrak{g})$ only. It is worthwhile noting that
although the retraction of de Rham complex on $d_c$ costs a lot of
derivatives, this is harmless here due to the feature of Sobolev
$L^{q,p}$ cohomology we mentionned above. Actually, in this low energy
large scale problem, loosing regularity is not an issue and the less
derivatives the homotopy $K_c$ controls, the smaller is the $(p,q)$
gap in Sobolev inequality, and the better becomes the $\ell^{q,p}$
vanishing result.

\smallskip

The non-vanishing result (ii) in Theorem~\ref{1} relies on the
construction of homogeneous closed differential forms of any order and
controlled weights. The contracted de Rham complex is useful to this
end too. Such homogeneous forms belong to $L^p$ space with explicit
$p$, but can not be integrated in $L^q$ for $q$ too close to $p$, as
seen using Poincaré duality (construction of compactly supported test
forms and integration by parts).

\section{Local Poincar\'e inequality}

In this section, differential forms of degree $-1$, $\Omega^{-1}$, are
meant to be constants. The complex is completed with the map
$d:\Omega^{-1}\to\Omega^0$ which maps a constant to a constant
function with the same value.

\begin{dfi}
  Let $M$ be a Riemannian manifold. Say that $M$ has $C^h$-bounded
  geometry if injectivity radius is bounded below and curvature
  together with all its derivatives up to order $h$ are uniformly
  bounded.  For $\ell\leq h-1$, let $W^{\ell,p}$ denote the space of
  smooth functions $u$ on $M$ which are in $L^p$ as well as all their
  covariant derivatives up to order $\ell$. When $p=\infty$,
  $W^{\ell,\infty}=C^{\ell}\cap L^{\infty}$.
\end{dfi}

\begin{rem}\label{whpdecr}
  On a $C^h$-bounded geometry $n$-manifold, these Sobolev spaces are
  interlaced. Let $q\geq p$. Then $W^{h,p}\subset W^{h-1-n/p,q}$.
\end{rem}
Indeed, on a ball $B$ of size smaller than the injectivity radius,
usual Sobolev embedding holds, $W^{\ell,p}(B)\subset L^q(B)$ provided
$\frac{1}{p}-\frac{1}{q}\leq \frac{\ell}{n}$, an inequality which is
automatically satisfied if $\ell\geq \frac{n}{p}+1$. Pick a covering
$B_i$ of $M$ by such balls with bounded multiplicity. Let
$u_i=\|f\|_{W^{h,p}(B)}$ and $v_i=\|f\|_{W^{h-\ell,q}(B)}$. Then
$v_i\leq C\,u_i$. Furthermore
\begin{eqnarray*}
  \|f\|_{W^{h-\ell,q}}\leq C\,\|(v_i)\|_{\ell^q}\leq C\,\|(v_i)\|_{\ell^p}\leq C'
  \|(u_i)\|_{\ell^p}\leq C''\,\|f\|_{W^{h,p}}.
\end{eqnarray*}

\begin{rem}
  According to Bemelmans-Min Oo-Ruh \cite{B-MinOo-Ruh}, for any fixed
  $h$, any complete Riemannian metric with bounded curvature can be
  approximated by an other one with all derivatives of curvature up to
  order $h$ uniformly bounded.
\end{rem}
Therefore assuming bounded geometry up to a high order is not a
restriction for our overall purposes. The main point is that curvature
and injectivity radius be bounded. Nevertheless, it helps in technical
steps like the following Proposition.

\begin{pro}\label{local}
  Let $M$ be a Riemannian manifold with sectional curvature bounded by
  $K$, as well as all covariant derivatives of curvature up to order
  $h$. Let $R<\frac{\pi}{2\sqrt{K}}$. Assume that $M$ has a positive
  injectivity radius, larger than $2R$. Let $y\in M$. Let $U_j$ be
  balls in $M$ containing $y$ of radii $\leq R$, and
  $U=\bigcap_j U_j$. The Cartan homotopy is an operator $P$ on
  differential forms on $U$ which satisfies $1=Pd+dP$ and maps
  $C^{\ell}\Omega^{k}(U)$ to $C^{\ell}\Omega^{k-1}(U)$ for all
  $\ell\leq h-1$ and all $k=0,\ldots,n$, with norm depending on $K$,
  $h$ and $R$ only.

  Assume further that $U$ contains $B(y,r)$ for some $r>0$. Then $P$
  is bounded from $W^{h-1,p}\Omega^{k}(U)$ to
  $W^{h-n-1,q}\Omega^{k-1}(U)$, provided $p\geq 1$, $q\geq 1$,
  $h>n+1$. Its norm depends on $K$, $h$, $R$ and $r$ only.
\end{pro}

\begin{proof}
  The assumptions on $R$ guarantee that minimizing geodesics between
  points at distance $<R$ are unique and that all balls of radii
  $\leq R$ are geodesically convex. For $x,y\in U$, let $\gamma_{x,y}$
  denote the unique minimizing geodesic from $x$ to $y$, parametrized
  on $[0,1]$ with constant speed $d(x,y)$. Fix $y\in U$. Consider the
  vectorfield $\xi_y$ defined as follows,
  \begin{eqnarray*}
    \xi_{y}(x)=\gamma'_{x,y}(0).
  \end{eqnarray*}
  It is smooth, since in normal coordinates centered at $y$, $\xi_{y}$
  is the radial vectorfield $\xi_y(u)=-u$. Let $\phi_{y,t}$ denote the
  diffeomorphism semi-group generated by $\xi_y$. For $t\in\R_+$,
  $\phi_{y,t}$ maps a point $x$ to $\gamma_{y,x}(e^{-t}d(y,x))$.

  Let $k\geq 1$. Following H. Cartan, define an operator $P_y$ on $k$
  forms $\omega$ by
  \begin{eqnarray*}
    P_y(\omega)=-\int_{0}^{+\infty}\phi_{y,t}^{*}\iota_{\xi_{y}}\omega\,dt.
  \end{eqnarray*}
  Then, on $k$-forms,
  \begin{eqnarray*}
    dP_y+P_yd &=&- \int_{0}^{+\infty}\phi_{y,t}^{*} (d\iota_{\xi_{y}} +
                  \iota_{\xi_{y}}d)\,dt\\    
              &=&-\int_{0}^{+\infty}\phi_{y,t}^{*}\mathcal{L}_{\xi_{y}}\,dt\\
              &=&-\phi_{y,+\infty}^{*}+\phi_{y,0}^{*}=1,
  \end{eqnarray*}
  if $k\geq 1$, since $\phi_{y,0}$ is the identity and
  $\phi_{y,+\infty}$ is the constant map to $y$. On $0$-forms, define
  $P_y(\omega)=\omega(y)$. Then $dP_y=\omega(y)$ viewed as a
  (constant) function on $U$, whereas $P_y d\omega=\omega-\omega(y)$,
  hence $dP_y+P_y d=1$ also on $0$-forms.

  In normal coordinates with origin at $y$, $P_y$ has a simple
  expression
  \begin{eqnarray*}
    P_y\omega(x)&=&\int_{0}^{+\infty}e^{-kt}\omega_{e^{-t}x}(x,\cdots)\,dt\\
                &=&\int_{0}^{1}s^{k-1}\omega_{sx}(x,\cdots)\,ds.
  \end{eqnarray*}
  It shows that $P_y$, read in normal coordinates, is bounded on
  $C^{\ell}$ for all $\ell$.

  The domain of exponential coordinates, $V=exp_y^{-1}(U)$, is
  convex. If it contains a ball of radius $r$, there is a bi-Lipschitz
  homeomorphism of the unit ball to $V$ with Lipschitz constants
  depending on $R$ and $r$ only, hence Sobolev embeddings
  $W^{1,p}\subset L^q$ for $\frac{1}{p}-\frac{1}{q}\leq\frac{1}{n}$
  with uniform constants (if $q=\infty$, $p>n$ is required and $L^q$
  is replaced with $C^0$). If $p\geq 1$, this implies that
  $W^{n+1,p}\subset C^0$, hence $W^{\ell,p}\subset
  C^{\ell-n-1}$. Obviously, $C^{\ell-n-1}\subset W^{\ell-n-1,q}$.

  Since curvature and its derivatives are bounded up to order $h$, the
  Riemannian exponential map and its inverse are $C^{h-1}$-bounded,
  hence $P_y$ is bounded on $C^{\ell}$ for $\ell\leq h-1$. If $h>n+1$,
  the embeddings
  \begin{eqnarray*}
    W^{\ell,p}\subset C^{\ell-n-1}\subset W^{\ell-n-1,q}
  \end{eqnarray*}
  hold on $U$ with bounds depending on $K$, $R$ and $r$ only, hence
  $P$ maps $W^{\ell,p}$ to $W^{\ell-n,q}$, with uniform bounds.
\end{proof}

\section{$\ell^{q,p}$ cohomology and Sobolev $L^{q,p}$ cohomology}
\label{sec:leray}

\begin{dfi}
  Let $M$ be a Riemannian manifold. Let $h,h'\in\N$. The \emph{Sobolev
    $L^{q,p}$ cohomology} is
  \begin{eqnarray*}
    L_{h',h}^{q,p}H^{\cdot}(M)=\{\text{closed forms in }W^{h,p}\} /
    d(\{\text{forms in }W^{h',q}\}) \cap W^{h,p}.
  \end{eqnarray*} 
\end{dfi}

\begin{rem}
  \label{rem:decreasing}
  On a bounded geometry $n$-manifold, this is nonincreasing in $q$ in
  the following sense. Let $q'\geq q$. Then $L_{h',h}^{q,p}H^{k}(M)$
  surjects onto $L_{h'-qn-1,h}^{q',p}H^{k}(M)$, see Remark
  \ref{whpdecr}.
\end{rem}

\begin{thm}
  \label{prop:leray}
  Let $1\leq p\leq q\leq \infty$. Let $M$ be a Riemannian manifold of
  $C^\ell$-bounded geometry, with $\ell > n^{n+1}+1$. Let $T$ be a
  simplicial complex quasiisometric to $M$. For every integers $h,h'$
  such that $n^{n+1}<h,h'\leq \ell -1$, there exists an isomorphism
  between the $\ell^{q,p}$ cohomology of $T$ and the Sobolev $L^{q,p}$
  cohomology $L_{h',h}^{q,p}H^{\cdot}(M)$.
\end{thm}

The proof is a careful inspection of Leray's acyclic covering theorem.

First construct a simplicial complex $T$ quasiisometric to $M$. Pick a
left-invariant metric on $M$. Up to rescaling, one can assume that
sectional curvature is $\leq 1/n^2$ and injectivity radius is
$>2n$. Pick a maximal $1/2$-separated subset $\{x_i\}$ of $M$. Let
$B_i$ be the covering by closed unit balls centered on this set. Let
$T$ denote the nerve of this covering. Let $U_i=B(x_i,3)$. Note that
if $x_{i_0},\ldots,x_{i_j}$ span a $j$-simplex of $T$, then the
intersection
\begin{eqnarray*}
  U_{{i_0}\ldots{i_j}} := \bigcap_{m=0}^{j}U_{i_m}
\end{eqnarray*}
is contained in a ball of radius 3 and contains a concentric ball of
radius 1.

Pick once and for all a smooth cut-off function with support in
$[-1,1]$, compose it with distance to points $x_i$ and convert the
obtained collection of functions into a partition of unity $\chi_i$ by
dividing by the sum.

Define a bicomplex $C^{j,k}=$ skew-symmetric maps associating to
$j+1$-tuples $(i_0,\ldots, i_j)$ differential $k$-forms on $j+1$-fold
intersections $U_{i_0}\cap\cdots\cap U_{i_j}$. It is convenient to
extend the notation to
\begin{eqnarray*}
  C^{-1,k} =\Omega^k(M),\quad C^{j,-1}=C^j(T),\quad
  C^{j,\cdot}=C^{\cdot,j}=0 \textrm{ if }  j<-1.
\end{eqnarray*}

The two commuting complexes are $d:C^{j,k}\to C^{j,k+1}$ and the
simplicial coboundary $\delta: C^{j,k}\to C^{j+1,k}$ defined by
\begin{eqnarray*}
  \delta(\phi)_{i_0 \ldots i_{j+1}}= \phi_{i_0 \ldots i_{j}}-\phi_{i_0
  \ldots i_{j-1}j_{j+1}} +\cdots+(-1)^{j+1}\phi_{i_1 \ldots i_{j+1}},
\end{eqnarray*}
restricted to $U_{i_0}\cap\cdots\cap U_{i_{j+1}}$. By convention,
$d:C^{j,-1}\to C^{j,0}$ maps scalar $j$-cochains to skew-symmetric
maps to functions on intersections which are constant. Also,
$\delta:C^{-1,k}\to C^{0,k}$ maps a globally defined differential form
to the collection of its restrictions to open sets $U_i$. Other
differentials vanish.

The coboundary $\delta$ is inverted by the operator
\begin{eqnarray*}
  \epsilon:C^{j,k}\to C^{j-1,k}
\end{eqnarray*}
defined by
\begin{eqnarray*}
  \epsilon(\phi)_{i_0 \ldots i_{j-1}}=\sum_{m}\chi_m \phi_{m i_0 \ldots i_{j-1}}.
\end{eqnarray*}
By an inverse, we mean that $\delta\epsilon+\epsilon\delta=1$. This
identity persists in all nonnegative bidegrees provided
$\epsilon : C^{0,k} \to C^{-1,k}$ is defined by
$\epsilon (\phi) = \sum_m \chi_m \phi_m$ and $\epsilon=0$ on
$C^{j,k}$, $j<0$.

Consider the maps
$$
\Phi_j = (\epsilon d)^{j+1} : C^j(T)=C^{j,-1} \to
C^{-1,j}=\Omega^j(M).
$$ 
By definition, given a cochain $\kappa$, $\Phi_j(\kappa)$ is a local
linear expression of the constants defining $\kappa$, multiplied by
polynomials of the $\chi_i$ and their differential. Therefore
$\Phi_j(\kappa)$ is $C^\infty$ and belongs to $W^{h,q}(\Omega^j(M))$
for all $q \geq p\geq 1$ if $\kappa \in \ell^p(C^j(T))$.

Since
\begin{eqnarray*}
  (\epsilon d)\delta=\epsilon\delta d=(1-\delta\epsilon)d=d-\delta\epsilon d,
\end{eqnarray*}
for $j\geq 0$, on $C^{j,-1}$,
\begin{eqnarray*}
  (\epsilon d)^{j+2}\delta
  &=&(\epsilon d)^{j+1} (d-\delta\epsilon d)= -(\epsilon d)^{j+1}
      \delta(\epsilon d) \\
  &=&\cdots\\
  &=&(-1)^{j+1}(\epsilon d)\delta(\epsilon
      d)^{j+1}=(-1)^{j+1}(d-\delta\epsilon d) (\epsilon d)^{j+1}\\
  &=&(-1)^{j+1}d(\epsilon d)^{j+1}-(-1)^{j+1}\delta(\epsilon d)^{j+2}\\
  &=&(-1)^{j+1}d(\epsilon d)^{j+1}.
\end{eqnarray*}
Indeed, $(\epsilon d)^{j+2}(C^{j,-1})\subset C^{-2,j+1}=\{0\}$. In
other words,
\begin{eqnarray*}
  \Phi_{j+1} \circ \delta=(-1)^{j+1}d\circ\Phi_j.
\end{eqnarray*}

\medskip

We now proceed in the opposite direction to produce a cohomological
inverse of $\Phi$. Let us mark each intersection
$U_{i_0 \ldots i_j}:=U_{i_0}\cap\cdots\cap U_{i_j}$ with the point
$y=x_{i_0}$. Proposition \ref{local} provides us with an operator
$P_{i_0\ldots i_j}$ on usual $k$-forms on $U_{i_0 \ldots i_j}$.
Putting them together yields an operator $P:C^{j,k}\to C^{j,k-1}$ such
that $1=dP+Pd$.  Furthermore, $P$ is bounded from $W^{\cdot,p}$ to
$W^{\cdot-n-1,q}$.

Exchanging the formal roles of $(\delta, \epsilon)$ and $(d,P)$, we
define
\begin{displaymath}
  \Psi_k = (P\delta)^{k+1} :\Omega^k(M)=C^{-1,k}\to C^{k,-1}=C^k(T), 
\end{displaymath}
As above, one checks easily that
$\Psi_{k+1} \circ d = (-1)^{k+1}\delta \circ \Psi_k$.

Observe that the maps $\Psi_j \circ \Phi_j$, $j=0,1,\ldots$, put
together form a morphism of the complex $C^{\cdot,-1}=C^\cdot(T)$
(i.e. they commute with $\delta$). We next show that it is homotopic
to the identity. Let us prove by induction on $i$ that, on
$C^{\cdot,-1}$,
\begin{equation}
  \label{eq:1}
  (P \delta)^i (\epsilon d)^i = 1 - R_i \delta - \delta R_{i-1}\,,
\end{equation}
with $R_0= 0$ and
$R_i = \sum_{k=0}^{i-1} (-1)^k (P \delta)^k P (\epsilon d)^{k+1}$.
This implies the result for $\Psi_j \circ \Phi_j$.
\begin{proof}
  For $i=1$, one has on $C^{\cdot,-1}$
  \begin{align*}
    (P \delta) (\epsilon d) & = P(1 - \epsilon \delta)d = Pd -
                              P\epsilon  d \delta \\
                            & = 1 - (P\epsilon d) \delta ,
  \end{align*}
  since $Pd = 1$ on $C^{\cdot,-1}$. Assuming \eqref{eq:1} for $i$, one
  writes
  \begin{align}
    (P\delta)^{i+1} (\epsilon d)^{i+1} 
    & = (P \delta)^i P\delta
      \epsilon d (\epsilon d)^i =  (P \delta)^i P(1 - \epsilon
      \delta) d (\epsilon d)^i  \nonumber \\
    & = (P
      \delta)^i (1 - dP - P \epsilon \delta d ) (\epsilon d)^i
      \nonumber \\
    & = (P \delta)^i  (\epsilon d)^i -  (P \delta)^i dP  (\epsilon
      d)^i - (P \delta)^i P (\epsilon d) \delta (\epsilon
      d)^i .
      \label{eq:2}
  \end{align}
  
  About the second term in \eqref{eq:2}, one finds that
  $\delta(P\delta) d = - \delta d (P\delta)$, hence by induction, one
  can push the isolated $d$ term to the left,
  \begin{displaymath}
    (P \delta)^i d = (-1)^{i-1} P\delta d (P\delta)^{i-1} = (-1)^{i-1}
    Pd \delta (P\delta)^{i-1} = (-1)^{i-1}\delta (P\delta)^{i-1} ,
  \end{displaymath}
  when the image lies within $C^{\cdot,-1}$, as it does in
  \eqref{eq:2}.
  
  For the third term in \eqref{eq:2}, one sees that
  $(\epsilon d) \delta (\epsilon d) = (\epsilon d)(1 - \epsilon
  \delta) d = - (\epsilon d)^2 \delta$,
  so that one can push the isolated $\delta$ term to the right,
  \begin{displaymath}
    (P \delta)^i P (\epsilon d) \delta (\epsilon
    d)^i =  (-1)^i (P \delta)^i P (\epsilon d)^{i+1} \delta \,.
  \end{displaymath}
  Gathering in \eqref{eq:2} gives
  \begin{displaymath}
    (P\delta)^{i+1} (\epsilon d)^{i+1}  = 1 - (R_i + (-1)^i
    (P\delta)^i P (\epsilon d)^{i+1} ) \delta - \delta (R_{i-1} +
    (-1)^{i-1}  (P\delta)^{i-1} P (\epsilon d)^i),
  \end{displaymath}
  that proves \eqref{eq:1}.
\end{proof}

Similarly, $\Phi\circ\Psi$ is homotopic to the identity on the complex
$(C^{-1,\cdot},d)$, $\Phi\circ\Psi=1-dR'-R'd$.

Finally, let us examine how Sobolev norms behave under the class of
endomorphisms we are using. Maps from cochains to differential forms,
i.e. $\epsilon$, $\Phi$ are bounded from $\ell^p$ to $W^{h-1,p}$. Maps
from differential forms, i.e. $P$, $\Psi$, loose derivatives (but this
is harmless since the final outputs are scalar cochains) so are
bounded from $W^{h,p}$ to $\ell^p$, $h\leq \ell-1$. One merely needs
$h$ large enough to be able to apply $P$ $n$ times, whence the
assumption $h\geq n^{n+1}$. Maps from cochains to cochains, e.g. $R$,
are bounded on $\ell^p$, maps from differential forms to differential
forms, e.g. $R'$, are bounded from $W^{h',p}$ to $W^{h,p}$ for every
$h'\geq n^{n+1}$ such that $h'\leq \ell-1$.

If $q\geq p$, the $\ell^q$-norm is controlled by the $\ell^p$-norm,
hence $R$ is bounded from $\ell^p$ to $\ell^q$. It is also true that
$R'$ is bounded from $W^{h'-1,p}$ to $W^{h-1,q}$. Indeed, it is made
of bricks which map differential forms to cochains or cochains to
differential forms, so no loss on derivatives affects the final
differentiability. For the same reason, one can gain local
integrability from $L^p_{loc}$ to $L^q_{loc}$ without restriction on
$p$ and $q$ but $p$, $q\geq 1$.

It follows that $\Phi$ and $\Psi$ induce isomorphisms between the
$\ell^{q,p}$ cohomology of $T$ and the Sobolev $L^{q,p}$ cohomology.

\section{De Rham complex and graduation on Carnot groups}
\label{sec:dc_complex}

From now on, we will work on Carnot Lie groups. These are nilpotent
Lie groups $G$ such that their Lie algebra $\mathfrak{g}$ splits into
a direct sum
$$
\mathfrak{g} = \mathfrak{g}_1 \oplus \mathfrak{g}_2 \oplus \cdots
\oplus\mathfrak{g}_r \quad \mathrm{satisfying} \quad [\mathfrak{g}_1,
\mathfrak{g}_i] = \mathfrak{g}_{i+1} \ \mathrm{for} \ 1\leq i \leq
r-1.
$$
The weight $w = i$ on $\mathfrak{g}_i$ induces a family of dilations
$\delta_t = t^w$ on $\mathfrak{g} \simeq G$.

In turn, the tangent bundle $TG$ splits into left invariant
sub-bundles $H_1 \oplus \cdots \oplus H_r$ with $H_i = \mathfrak{g}_i$
at the origin. Finally, differential forms decompose through their
weight $\Omega^k G = \oplus_w \Omega^{k,w} G$ with
$\Omega^{k_1} H^*_{w_1}\wedge \cdots \wedge \Omega^{k_i}H^*_{w_i}$ of
weight $w = k_1 w_1 + \cdots + k_i w_i$. De Rham differential $d$
itself splits into
$$
d = d_0 + d_1 + \cdots + d_r,
$$ with $d_i$ increasing weight by $i$. Indeed, this is clear on functions
where $d_0 = 0$ and $d_i = d $ along $H_i$. This extends to forms,
using $d(f\alpha) = df \wedge \alpha + f d \alpha$ and observing that
for left invariant forms $\alpha$ and left invariant vectors $X_i$,
Cartan's formula reads
\begin{align*}
  d \alpha (X_1,\cdots , X_{k+1}) & = \sum_{1\leq i < j\leq k+1} (-1)^{i+j}
                                    \alpha 
                                    ([X_i,X_j ], \cdots, \widehat{X}_{i,j},
                                    \cdots  X_{k+1}) \\
                                  & = d_0 \alpha (X_1,\cdots , X_{k+1}).
\end{align*}
Hence $d= d_0$ is a weight preserving algebraic (zero order) operator
on invariant forms, and over a point,
$\ker d_0 / \im d_0 = H^*(\mathfrak{g})$ is the Lie algebra cohomology
of $G$. Note also that from these formulas, $d_i$ is a homogeneous
differential operator of degree $i$ and increases the weight by
$i$. It is homogeneous of degree $0$ trough $h^*_\lambda$, as $d$,
since $d h^*_\lambda = h^*_\lambda d$.

\medskip

This algebraic $d_0$ allows to split and contract de Rham complex on a
smaller subcomplex, as we now briefly describe. This was shown in
\cite{Rumin_CRAS_1999} and \cite{Rumin_TSG} in the more general
setting of Carnot--Caratheodory manifolds. More details may be found
there.

Pick an invariant metric so that the $\mathfrak{g}_i$ are orthogonal
to each others, and let $\delta_0 = d_0^*$ and $d_0^{-1}$ be the
partial inverse of $d_0$ such that $\ker d_0^{-1} = \ker \delta_0$,
$d_0^{-1} d_0 = \Pi_{\ker \delta_0}$ and
$d_0 d_0^{-1} = \Pi_{\im d_0}$.

Let
$E_0 = \ker d_0 \cap \ker \delta_0 \simeq \Omega^* H^*(\mathfrak{g})$.
Iterating the homotopy $r = 1 - d_0^{-1} d - d d_0 ^{-1}$ one can show
the following results, stated here in the particular case of Carnot
groups.

\begin{thm} \cite[Theorem~1]{Rumin_CRAS_1999}
  \label{thm:dc-complex}
  \begin{enumerate}
  \item The de Rham complex on $G$ splits as the direct sum of two
    sub-complexes $E \oplus F$, where
    $E = \ker d_0^{-1} \cap \ker d d_0^{-1}$ and
    $F= \im d_0^{-1} + \im (d d_0^{-1}) $.
  \item The retractions $r^k$ stabilize to $\Pi_E$ the projection on
    $E$ along $F$. $\Pi_E$ is a homotopy equivalence of the form
    $\Pi_E = 1 - Rd - dR$ where $R$ is a differential operator.
  \item One has $\Pi_E \Pi_{E_0} \Pi_E = \Pi_E$ and
    $\Pi_{E_0} \Pi_E \Pi_{E_0} = \Pi_{E_0}$ so that the complex
    $(E,d)$ is conjugated through $\Pi_{E_0}$ to the complex
    $(E_0, d_c)$ with $d_c = \Pi_{E_0} d \Pi_E \Pi_{E_0}$.
  \end{enumerate}

\end{thm}

This shows in particular that the de Rham complex, $(E,d)$ and
$(E_0, d_c)$ are homotopically equivalent complexes on smooth
forms. For convenience in the sequel, we will refer to $(E_0,d_c)$ as
the \emph{contracted de Rham complex} (also known as Rumin complex)
and sections of $E_0$ as \emph{contracted forms}, since they have a
restricted set of components with respect to usual ones.

We shall now describe its analytical properties we will use.

\section{Inverting $d_c$ and $d$ on $G$}
\label{sec:inverting-d-d_c}

De Rham and contracted de Rham complexes are not homogeneous as
differential operators, but are indeed invariant under the dilations
$\delta_t$ taking into account the weight of forms. This leads to a
notion of sub-ellipticity in a graded sense, called C-C ellipticity in
\cite{Rumin_CRAS_1999, Rumin_TSG}, that we now describe.

Let $\nabla = d_1$ the differential along $H = H_1$. Extend it on all
forms using $\nabla (f \alpha) = (\nabla f ) \alpha$ for left
invariant forms $\alpha$ on $G$. Kohn's Laplacian
$\Delta_H = \nabla^* \nabla$ is hypoelliptic since $H$ is bracket
generating on Carnot groups, and positive self-adjoint on $L^2$. Let
then
$$
|\nabla| = \Delta_H^{1/2}
$$
denotes its square root. Following \cite[Section~3]{Folland_1975} or
\cite{CGGP}, it is a homogeneous first order pseudodifferential
operator on $G$ in the sense that its distributional kernel, acting by
group convolution, is homogeneous and smooth away from the origin.  It
possesses an inverse $|\nabla|^{-1}$, which is also a homogeneous
pseudodifferential operator of order $-1$ in this calculus. Actually
according to \cite[Theorem 3.15]{Folland_1975}, it belongs to a whole
analytic family of pseudodifferential operators $|\nabla|^\alpha$ of
order $\alpha \in \C$. Note that kernels of these homogeneous
pseudodifferential operators may contain logarithmic terms, when the
order is an integer $\leq -Q$. We refer to \cite{CGGP} and
\cite[Section~1]{Folland_1975} for more details and properties of this
calculus.

\smallskip A particularly useful test function space for these
operators is given by the space of Schwartz functions all of whose
polynomial moments vanish,
\begin{equation}
  \label{eq:3}
  \S_0=\{f\in\S\,;\,\langle f,P\rangle=0 \text{ for every
    polynomial }P\}, 
\end{equation}
where $\mathcal{S}$ denotes the Schwartz space of $G$ and
$\langle f,P\rangle=\int_Gf(x)P(x)\,dx$. Unlike more usual test
functions spaces as $C^\infty_c$ or $\mathcal{S}$, this space $\S_0$
is stable under the action of pseudodifferential operators of any
order in the calculus, see \cite[Proposition 2.2]{CGGP}, so that they
can be composed on it. In particular by \cite[Theorem
3.15]{Folland_1975}, for every $\alpha$, $\beta \in \C$,
\begin{equation}
  \label{eq:4}
  |\nabla|^\alpha |\nabla|^{\beta} = |\nabla|^{\alpha+\beta} \
  \mathrm{on} \ \S_0 \,.
\end{equation}
We shall prove in Proposition~\ref{prop:density-s_0} that $\S_0$ is
dense in all Sobolev spaces $W^{h,p}$ with $h \in \N$ and
$1< p< +\infty$, but we shall work mainly in $\S_0$ in this section.

\medskip

Now let $N=w$ on forms of weight $w$. Consider the operator
$|\nabla|^N$, preserving the degree and weight of forms, and acting
componentwise on $\S_0$ in a left-invariant frame.

From the previous discussion,
$d^{\nabla} = |\nabla|^{-N} d |\nabla|^N$ and
$d_c^{\nabla} = |\nabla|^{-N} d_c |\nabla|^N$ are both homogeneous
pseudodifferential operators of (differential) order $0$. Indeed, as
observed in Section~\ref{sec:dc_complex}, $d$ splits into
$d=\sum_i d_i$ where $d_i$ is a differential operator of horizontal
order $i$ which increases weight by $i$.  On forms of weight $w$,
$|\nabla|^N$ has differential order $w$, $d_i |\nabla|^N$ has order
$w+i$ and maps to forms of weight $w+i$, on which $ |\nabla|^{-N}$ has
order $-(w+i)$, hence $|\nabla|^{-N} d |\nabla|^N$ has order $0$. The
same argument applies to $d_c$.

Viewed in this Sobolev scale, these complexes become invertible in the
pseudodifferential calculus. Let
$$
\Delta^{\nabla } = d^{\nabla} (d^{\nabla})^* + (d^{\nabla})^* d^\nabla
\quad \mathrm{and} \quad \Delta_c^{\nabla} = d_c^{\nabla}
(d_c^{\nabla})^* + (d_c^{\nabla})^* d_c^\nabla,
$$
not to be confused with the non homogeneous
$d_c^{\nabla} (d_c^*)^{\nabla} + (d_c^*)^{\nabla} d_c^{\nabla} =
|\nabla|^{-N} (d_c d_c^* + d_c^* d_c ) |\nabla|^N$.

\begin{thm}\label{thm:left_inverse}
  \cite[Theorem 3]{Rumin_CRAS_1999}, \cite[Theorem 5.2]{Rumin_TSG} The
  Laplacians $\Delta^{\nabla} $ and $\Delta_c^{\nabla } $ have left
  inverses $Q^{\nabla}$ and $Q_c^{\nabla}$, which are zero order
  homogeneous pseudodifferential operators.
\end{thm}

By \cite[Theorem 6.2]{CGGP}, this amounts to show that these
Laplacians satisfy Rockland's injectivity criterion. This means that
their symbols are injective on smooth vectors of any non trivial
irreducible unitary representation of $G$.


\medskip

This leads to a global homotopy for $d_c$ on $G$. Indeed following
\cite[Proposition~1.9]{Folland_1975}, homogeneous pseudodifferential
operators of order zero on $G$, such as $d_c^{\nabla}$,
$\Delta_c^\nabla$ and $Q_c^{\nabla}$, are bounded on all $L^p(G)$
spaces for $1< p< \infty$.  Therefore the positive self-adjoint
$\Delta_c^{\nabla }$ on $L^2(G)$ is bounded from below since
$Q_c^{\nabla} \Delta_c^{\nabla} = 1$.  Hence, it is invertible in
$L^2(G)$ and $Q_c^{\nabla} = (\Delta_c^{\nabla})^{-1}$ is the inverse
of $\Delta_c^{\nabla}$.

Since $(\Delta_c^{\nabla})^{-1}$ commute with $d_c^{\nabla}$, the zero
order homogeneous pseudodifferential operator
$K_c^\nabla = (d_c^{\nabla})^* (\Delta_c^{\nabla})^{-1} $ is a global
homotopy,
\begin{displaymath}
  1 = d_c^\nabla K_c^\nabla + K_c^\nabla d_c^\nabla.
\end{displaymath}
Let us set
\begin{displaymath}
  K_c = |\nabla|^{N}K_c^\nabla|\nabla|^{-N} 
\end{displaymath}
on $\S_0$, in order that
\begin{displaymath}
  1 = d_c K_c + K_c d_c.
\end{displaymath}
Since $K_c^\nabla$ is bounded from $L^p$ to $L^p$ for $1<p< \infty$,
$K_c$ is bounded on $\mathcal{S}_0$ endowed with the graded Sobolev
norm
\begin{displaymath}
  \|\alpha\|_{|\nabla|,N,p}:=\||\nabla|^{-N}\alpha\|_p \,.
\end{displaymath} 
Actually $K_c$ stays bounded for the whole Sobolev scale with shifted
weights $N - m$.
\begin{pro}
  \label{prop:K_c_scale}
  For every constant $m\in\R$ and $1< p< +\infty$, the homotopy $K_c$
  is also bounded on $\S_0$ with respect to the norm
  $\| \ \|_{|\nabla|,N-m,p}$.
\end{pro}

\begin{proof}
  Indeed for $\alpha\in \mathcal{S}_0$,
  \begin{align*}
    \|K_c \alpha\|_{|\nabla|,N-m,p}
    &= \||\nabla|^{m-N}K_c\alpha\|_{p}\\
    &= \||\nabla|^m | \nabla|^{-N}K_c| \nabla|^{N} |\nabla|^{-m}
      |\nabla|^{m-N}\alpha\|_{p} \quad \mathrm{by}\ \eqref{eq:4},\\
    &=\||\nabla|^m K_c^\nabla|\nabla|^{- m}| 
      \nabla|^{m-N}\alpha\|_{p} \\
    &\leq C\, \|| \nabla|^{m-N}\alpha\|_{p}
      =C\, \|\alpha\|_{|\nabla|,N-m,p} \,,
  \end{align*}
  since $|\nabla|^m K_c^\nabla|\nabla|^{- m}$ is pseudodifferential of
  order $0$ and homogeneous, hence bounded on $L^p$ if $1< p< \infty$.
\end{proof}

\begin{rem}
  \label{rem:1}
  Note that using Theorem~\ref{thm:left_inverse}, one can also produce
  a homotopy in the same way for the full de Rham complex itself. But
  as we will see in Section~\ref{sec:remarks-weight-gaps}, it leads to
  a weaker vanishing theorem for $\ell^{q,p}$ cohomology.
\end{rem}

\section{Relating the $|\nabla|$-graded to standard Sobolev norms}

The next step is to compare the $\|\ \|_{|\nabla|, N,p}$ norms,
depending on the weight of forms, to usual Sobolev norms of positive
order.

Fix a basis $X_i$ of $\mathfrak{g}_1$, viewed as left invariant
vectorfields on $G$. Let $\nabla$ denote the horizontal gradient
$\nabla f=(X_1 f,\ldots,X_{n_1} f)$. Acting componentwise on tuples of
functions allows to iterate it. One also extends it on differential
forms using their components in a left invariant basis. For
$1 \leq p \leq \infty$ and $h \in \N$, define the Sobolev $W_c^{h,p}$
norm
\begin{eqnarray*}
  \|\alpha\|_{W_c^{h,p}} = \sum_{k=0}^{h}\|\nabla^k\alpha\|_p \,.
\end{eqnarray*}

According to Folland \cite[Theorem~4.10,
Corollary~4.13]{Folland_1975}, these norms are equivalent to Sobolev
norms defined using $|\nabla|$, provided $1<p<+\infty$. Namely one has
then
\begin{equation}
  \label{eq:5}
  \|\ \|_{W_c^{h,p}} \asymp\sum_{k=0}^{h}\||\nabla|^k \  \|_p \,.
\end{equation}


We shall now compare these norms to the graded ones we introduced in
the previous section.

\begin{pro}
  \label{prop:Sobolev}
  Let $1< p< \infty$ be fixed, and $a, b, h \in \N$ be such that
  $a\leq b \leq a+h$. Let $\Omega_{[a,b]}$ denote the space of
  differential forms whose components have weights $a\leq w \leq b$.
  \begin{enumerate}[label=(\roman*)]
  \item It holds on $\Omega_{[a,b]}$ that
    \begin{displaymath}
      \sum_{m=b}^{a+h} \|\ \|_{|\nabla|, N-m, p} \leq
      \sum_{k=0}^{h}\||\nabla|^k \  \|_p  \leq
      \sum_{m=a}^{b+h} \|\ \|_{|\nabla|, N-m, p}\,. 
    \end{displaymath}
    
  \item Let $\mu \in \N$, $\mu<Q$, and $1< p< q< \infty$ satisfy
    $\frac{1}{p} - \frac{1}{q} = \frac{\mu}{Q}$. Then for some $C>0$
    it holds on $\Omega_{[a,b]} \cap \mathcal{S}$ that
    \begin{displaymath}
      C \| \ \|_{W^{h,q}_c} \leq  \sum_{m=a+ \mu}^{b+\mu + h} \| \
      \|_{|\nabla|, N-m,p} \,.
    \end{displaymath}
  \end{enumerate}
\end{pro}

\begin{proof}
  (i) By definition $a\leq N= w \leq b$ on $\Omega_{[a,b]}$, hence
  \begin{displaymath}
    \sum_{m=b}^{a+h} \|\ \|_{|\nabla|, N-m, p} = \sum_{m=b}^{a+h}
    \||\nabla|^{m-N}\ \|_p \leq \sum_{m=N}^{N+h}
    \||\nabla|^{m-N} \ \|_p = \sum_{k=0}^h
    \||\nabla|^k \ \|_p \,.
  \end{displaymath}
  One has also
  \begin{displaymath}
    \sum_{m=a}^{b+h} \|\ \|_{|\nabla|, N-m, p} = \sum_{m=a}^{b+h}
    \||\nabla|^{m-N}\ \|_p \geq \sum_{m=N}^{N+h}
    \||\nabla|^{m-N} \ \|_p = \sum_{k=0}^h
    \||\nabla|^k \ \|_p \,.
  \end{displaymath}

  \smallskip
  
  (ii) Since $|\nabla|^{-\mu}$ is a homogeneous pseudodifferential
  operator of order $-\mu$ and $\mu<Q$, its kernel is homogeneous of
  degree $-Q+\mu$. According to \cite[Proposition~1.11]{Folland_1975},
  $|\nabla|^{-\mu}$ is bounded from $L^{p}$ to $L^{q}$ if
  $1<p<q < +\infty$ satisfy
  $\frac{1}{p} - \frac{1}{q} = \frac{\mu}{Q}$, giving the
  Hardy-Littlewood-Sobolev inequality
  \begin{equation}
    \label{eq:6}
    \||\nabla|^{-\mu}\alpha\|_q \leq C \|\alpha \|_p \,. 
  \end{equation}
  Then for $\alpha \in \Omega_{[a,b]}\cap \mathcal{S}$
  \begin{align*}
    \sum_{m=a+ \mu}^{b+\mu + h} \| \alpha \|_{|\nabla|, N-m,p}
    & \geq  \sum_{m=N+\mu}^{N+\mu+ h} \||\nabla|^{m-N} \alpha \|_p 
      = \sum_{k=0}^h \| |\nabla|^{\mu + k} \alpha\|_p\\
    & \geq 1/C \sum_{k=0}^h \||\nabla|^k \alpha\|_q \asymp
      \|\alpha\|_{W^{h,p}_c}  
  \end{align*}
  using \eqref{eq:6}, \eqref{eq:5} and
  $|\nabla|^{\mu + k} = |\nabla|^\mu |\nabla|^k $ on $\mathcal{S}$ for
  $\mu,k\in \N$ as comes from \cite[Theorem~3.15]{Folland_1975}.
\end{proof}

\section{Density of $\S_0$ and extension of $K_c$}
\label{sec:density-s_0}

One knows by Proposition~\ref{prop:K_c_scale} that $K_c$ is continuous
with respect to the whole shifted norms $\|\ \|_{|\nabla|, N-m,p}$,
but it is only defined and provides an homotopy for $d_c$ on the
initial domain $\S_0$ so far.  Hence, we have to show that $\S_0$ is
dense in $\S$ for the standard Sobolev norms $\|\ \|_{W^{h,p}_c}$ in
order to extend $K_c$ on forms coming from $\ell^p$ cocycles in
Proposition~\ref{prop:leray}. Recall from \eqref{eq:3} that $\S_0$ is
the space of Schwartz functions with all vanishing polynomial
moments. It does not contain any non vanishing function with compact
support as seen using Fourier transform or Stone-Weierstrass
approximation theorem.

\begin{pro}
  \label{prop:density-s_0}
  For every $h\in\N$ and $1<p<+\infty$, $\S_0$ is dense in
  $W^{h,p}_c$. If $p=1$, $\S_0$ is dense in
  $\{f\in W_c^{h,1}\,;\,\langle f,1 \rangle=0\}$. If $p=\infty$,
  $\S_0$ is dense in the space $C^h_0(G)$ of functions of class $C^h$
  that tend to $0$ at infinity.
\end{pro}

\begin{proof}

  The group exponential map $\exp:\mathfrak{g}\to G$ maps Schwartz
  space to Schwartz space, polynomials to polynomials, and Lebesgue
  measure to Haar measure, hence $\S_0$ to $\S_0$. Pick coordinates
  adapted to the splitting
  $\mathfrak{g}=\mathfrak{g}_1\oplus\cdots\oplus\mathfrak{g}_s$.

  First we construct (on $\mathfrak{g}$) a family of functions
  $(g_\alpha)_{\alpha\in\N^n}$ which is dual to the monomial basis
  $(x^\beta)_{\beta\in\N^n}$ in the sense that
  \begin{eqnarray*}
    \langle g_\alpha,x^\beta\rangle=\begin{cases}
      1 & \text{if }\alpha=\beta, \\
      0  & \text{otherwise}.
    \end{cases}
  \end{eqnarray*}
  Let $\F$ denote the Euclidean Fourier transform on $\mathfrak{g}$,
  let $\Fi$ denote its inverse. Let $n=\mathrm{dim}(G)$, let
  $\alpha\in\N^n$ denote a multiindex,
  $|\alpha|=\sum_{i=1}^n \alpha_i$ its usual length,
  $\alpha!=\prod_{j=1}^n \alpha_j!$. Fix a smooth function $\chi$ with
  compact support in the square $\{\max |\xi_j| \leq 1\}$ of
  $\mathfrak{g}^*$, which is equal to $1$ in a neighborhood of
  $0$. Let
  \begin{eqnarray*}
    g_\alpha=\Fi(\chi(\xi)\frac{(i\xi)^{\alpha}}{\alpha!}).
  \end{eqnarray*}
  Then $g_\alpha\in\S$. One computes
  \begin{eqnarray*}
    (ix)^\beta
    g_\alpha(x)=  \Fi\left(\frac{\partial^{|\beta|}}{\partial\xi^{\beta}}
    (\chi(\xi) \frac{(i\xi)^\alpha}{\alpha!})\right).  
  \end{eqnarray*}
  Hence
  \begin{eqnarray*}
    \langle
    g_\alpha,x^\beta\rangle=(-i)^{|\beta|}
    \left( \frac{\partial^{|\beta|}}{\partial\xi^{\beta}} 
    (\chi(\xi)\frac{(i\xi)^\alpha}{\alpha!}) \right)(0)
    = (-i)^{|\beta|}\left( \frac{\partial^{|\beta|}}{\partial\xi^{\beta}}
    (\frac{(i\xi)^\alpha}{\alpha!})\right)(0)
  \end{eqnarray*}
  which vanishes unless $\beta=\alpha$, in which case it is equal to
  $1$.

  From
  \begin{eqnarray*}
    \frac{\partial g_{\alpha}}{\partial
    x^{\gamma}}=(-1)^{|\gamma|}\frac{(\alpha+\gamma)!}{\alpha!} g_{\alpha+\gamma},
  \end{eqnarray*}
  it follows that
  \begin{eqnarray*}
    (ix)^\beta \frac{\partial g_{\alpha}}{\partial x^{\gamma}} =
    (-1)^{|\gamma|}\frac{(\alpha+\gamma)!}{\alpha!}\Fi \left(
    \frac{\partial^{|\beta|}}{\partial \xi^\beta}
    (\chi(\xi)\frac{(i\xi)^{\alpha+\gamma}}{(\alpha+\gamma)!}) \right).
  \end{eqnarray*}
  Hence
  \begin{eqnarray*}
    \|x^\beta \frac{\partial g_{\alpha}}{\partial x^{\gamma}}
    \|_{\infty}\leq \frac{1}{\alpha!} \|\frac{\partial^{|\beta|}}{\partial
    \xi^\beta} (\chi(\xi)\xi^{\alpha+\gamma})\|_1.
  \end{eqnarray*}
  Thanks to Leibnitz' formula,
  \begin{eqnarray*}
    \frac{\partial^{|\beta|}}{\partial \xi^\beta} (\chi(\xi)
    \xi^{\alpha+\gamma}) 
    &=&\sum_{\eta+\eta'=\beta}
        \begin{pmatrix}
          \beta\\ \eta
        \end{pmatrix}
    \frac{\partial^{|\eta|}\chi}{\partial \xi^\eta}
    \frac{\partial^{|\eta'|}\xi^{\alpha +\gamma}}{\partial \xi^{\eta'}}\\
    &=&\sum_{\eta+\eta'=\beta}
        \begin{pmatrix}
          \beta\\\eta
        \end{pmatrix}
    \frac{\partial^{|\eta|}\chi}{\partial \xi^\eta}
    \frac{(\alpha+\gamma)!}{(\alpha+\gamma-\eta')!} \xi^{\alpha+\gamma-\eta'},
  \end{eqnarray*}
  where $\displaystyle
  \begin{pmatrix}
    \beta\\\eta
  \end{pmatrix}
  =\prod_{j=1}^n
  \begin{pmatrix}
    \beta_j\\\eta_j
  \end{pmatrix}$
  and, by convention, $\xi^{\alpha+\gamma-\eta'}=0$ if
  $\alpha +\gamma -\eta'\notin\N^n$. If $\beta$ and $\gamma$ are
  fixed, this is a polynomial in $\alpha$ of bounded degree $|\beta|$
  and bounded coefficients. So is its $L^1$ norm. Therefore
  \begin{eqnarray*}
    \|x^\beta \frac{\partial g_{\alpha}}{\partial x^{\gamma}}
    \|_{\infty}\leq  \frac{P_{\beta,\gamma}(\alpha)}{\alpha!},
  \end{eqnarray*}
  where $P_{\beta,\gamma}$ is a polynomial on $\R^n$.

  Let $w(j)$ denote the weight of the $j$-th basis vector. For
  $\beta\in\N^n$, let $w(\beta)=\sum_{j=1}^n \beta_j w(j)$. Let
  $g_{\alpha,t}=t^{w(\alpha)+Q}g_\alpha\circ\delta_t$. Then, for all
  $\beta$,
  $\langle x^\beta,g_{\alpha,t}\rangle=t^{w(\alpha)-w(\beta)}\langle
  x^\beta,g_\alpha\rangle$.
  Hence, for every choice of sequence $(t_\alpha)$, the family
  $(g_{\alpha,t_\alpha})_{\alpha\in\N^n}$ is again dual to the
  monomial basis. Also,
  \begin{eqnarray*}
    x^\beta \frac{\partial g_{\alpha,t}}{\partial x^{\gamma}} =
    t^{w(\alpha)-w(\beta)+w(\gamma)+Q} (x^\beta \frac{\partial
    g_{\alpha}}{\partial x^{\gamma}}) \circ \delta_t,
  \end{eqnarray*}
  hence, for every $1\leq p\leq \infty$ and $p'=\frac{p}{p-1}$,
  \begin{eqnarray*}
    \|x^\beta \frac{\partial g_{\alpha,t}}{\partial x^{\gamma}}\|_{p}
    = t^{w(\alpha)-w(\beta)+w(\gamma)+\frac{Q}{p'}} \|x^\beta 
    \frac{\partial g_{\alpha}}{\partial x^{\gamma}} \|_{p}.
  \end{eqnarray*}

  Given an arbitrary sequence $\mathbf{m}=(m_\alpha)_{\alpha\in\N^n}$,
  and a positive sequence $\mathbf{t}=(t_\alpha)_{\alpha\in\N^n}$,
  define the series
  \begin{eqnarray*}
    f_{\mathbf{m},\mathbf{t}}=\sum_{\alpha\in\N^n}m_\alpha g_{\alpha,t_\alpha}.
  \end{eqnarray*}
  If $t_\alpha$ and $|m_\alpha| t_\alpha^{w(\alpha)/2}$ stay bounded
  by $1$, then for every constant $C$,
  $|m_\alpha| t_\alpha^{w(\alpha)-C}$ stays bounded, the series
  converges in $\S$. Indeed, for every $\beta,\gamma\in\N^n$, the sum
  of $L^\infty$ norms
  $\|x^\beta \frac{\partial }{\partial x^{\gamma}}m_\alpha
  g_{\alpha,t_\alpha}\|_{\infty}$ is bounded above by
  \begin{eqnarray*}
    \sum_{\alpha\in\N^n}|m_\alpha|
    t_\alpha^{w(\alpha)-w(\beta)+w(\gamma)+Q}\,\frac{P_{\beta,\gamma}
    (\alpha)}{\alpha!} <\infty.
  \end{eqnarray*}
  By construction, all functions $f_{\mathbf{m},\mathbf{t}}$ have
  prescribed moments
  $\langle x^\alpha,f_{\mathbf{m},\mathbf{t}}\rangle=m_\alpha$.

  Fix a finite set $F$ of pairs $(\beta,\gamma)$ such that
  $w(\beta)\leq w(\gamma)$.  Let $W^{F,p}$ denote the completion of
  smooth compactly supported functions for the norm
  \begin{eqnarray*}
    \|f\|_{W^{F,p}}=\max\{\|x^\beta \frac{\partial }{\partial
    x^{\gamma}}m_\alpha g_{\alpha,t_\alpha} \|_{p}\,;\,(\beta,\gamma)\in F\}.
  \end{eqnarray*}
  Denote by
  \begin{eqnarray*}
    N^{F,p}_\alpha :=\max\{\|x^\beta \frac{\partial g_{\alpha}}{\partial
    x^{\gamma}}\|_{p} \,;\,(\beta,\gamma)\in F\}.
  \end{eqnarray*}
  Pick $\mathbf{t}$ such that, in addition to the previous
  assumptions, the series
  $\sum|m_\alpha|t_\alpha^{w(\alpha)+Q/p'}N^{F,p}_\alpha$
  converges. Then for every $\eps>0$, the series
  $f_{\mathbf{m},\eps\mathbf{t}}$ converges in $W^{F,p}$ and for
  $\eps \leq 1$
  \begin{eqnarray*}
    \|f_{\mathbf{m},\eps \mathbf{t}}\|_{W^{F,p}}\leq
    \eps^{Q/p'}|m_0|N^{F,p}_0 + \eps \sum_{\alpha\not=0}
    |m_\alpha|t_\alpha^{w(\alpha)+Q/p'} N^{F,p}_\alpha.
  \end{eqnarray*}
  Therefore, as $\eps$ tends to $0$, $f_{\mathbf{m},\eps\mathbf{t}}$
  tends to $0$ in $W^{F,p}$ (if $p=1$, one must assume that $m_0=0$).

  Given $f\in W^{F,p}$ (assume furthermore that
  $\langle f,1 \rangle=0$ if $p=1$ or that $f\in C_0^h(G)$ if
  $p=\infty$), approximate $f$ with an element $g\in\S$ (resp. such
  that $\langle g,1 \rangle=0$ if $p=1$). Set
  $m_\alpha=\langle x^\alpha,g \rangle$. Pick $\mathbf{t}$ satisfying
  the above smallness assumptions with respect to $\mathbf{m}$. Then
  $g-f_{\mathbf{m},\eps\mathbf{t}}\in \S_0$ and
  $f_{\mathbf{m},\eps\mathbf{t}}$ tends to $0$ in $W^{F,p}$, thus $f$
  belongs to the closure of $\S_0$.

  Finally, $\|f\|_{W_c^{h,p}}\leq \|f\|_{W^{F,p}}$ for a suitable
  finite set $F$. Indeed, $G$ admits a basis of left-invariant
  vectorfields $X_i$ of the form
  \begin{equation}
    \label{eq:7}
    X_i=\frac{\partial}{\partial x_i}+\sum_{j>i}P_{i,j}
    \frac{\partial}{\partial x_j}, 
  \end{equation}
  where $P_{i,j}$ is a $\delta_t$-homogeneous of weight
  $w(P_{i,j})<w(j)$.
\end{proof}

One can now extend $K_c$ from $\S_0$ to some Sobolev spaces, depending
on the weights on the source and the target. Let $W^{h,p}_c$ denotes
the completion of $\S$, and therefore $\S_0$, with respect to the norm
$\|\ \|_{W^{h,p}_c}$.

\begin{cor}
  \label{cor:K_c_Sobolev}
  Let $a\leq b\leq a+h$ and $a'\leq b'\leq a'+h'$ be integers.

  Suppose moreover that $1< p<q<\infty$ satisfy
  $\frac{1}{p} - \frac{1}{q} = \frac{\mu}{Q}$ with
  $0\leq \mu = b-a' < Q$ and $h= h' + b'-a' + b-a$.

  Let $(K_c)_{[a',b']}$ denotes the components of $K_c$ lying in
  $\Omega_{[a',b']}$. Then $K_c$ extends continuously on
  $\Omega_{[a,b]} \cap W^{h,p}_c$ so that
  \begin{displaymath}
    (K_c)_{[a',b']}(\Omega_{[a,b]} \cap W^{h,p}_c) \subset W^{h',q}_c \,.
  \end{displaymath}
\end{cor}

\begin{proof}
  Let $\alpha \in \Omega_{[a,b]}\cap \S_0$. Apply first
  Proposition~\ref{prop:Sobolev} (ii) to $(K_c)_{[a',b']}(\alpha)$
  with $a'$, $b'$, $\mu$ and $h'$ as above
  \begin{align*}
    C\|(K_c)_{[a',b']} (\alpha)\|_{W^{h',q}_c}
    & \leq \sum_{m= a' +
      \mu}^{b'+\mu+ h'} \|K_c
      \alpha\|_{|\nabla|, N-m,p} = \sum_{m= b}^{a+h} \|K_c
      \alpha\|_{|\nabla|, N-m,p} \\
    & \leq C'  \sum_{m= b}^{a+h} \|\alpha\|_{|\nabla|, N-m,p} \ \text{ by
      Proposition~\ref{prop:K_c_scale},}\\
    & \leq C'' \|\alpha \|_{W^{h,p}_c} 
  \end{align*}
  by Proposition~\ref{prop:Sobolev} (i) since
  $\alpha \in \Omega_{[a,b]}$.
\end{proof}

\begin{rem}
  \label{rem:K_c_formes_lisses}
  The statement becomes simpler when forms are smooth, meaning in
  $W^{+\infty,p}$, as those coming from $\ell^p $-cocycles on
  $G$. Namely in that case, one can let $h, h'$ go to $+\infty$ with
  $b'=Q$, $a=0$. Then $K_c$ integrates smooth forms in $W^{+\infty,p}$
  and of weight lower than $b$, into smooth forms whose components of
  weight larger than $a'$ lies in $W^{+\infty,q}$. Note that the
  $(p,q)$ gap is only determined by the weight gap $\mu = b-a'$ using
  Sobolev rule.
\end{rem}

\section{Proof of Theorem \ref{1}(i)}

Fix a left-invariant Riemannian metric $\mu$ on $G$. Let $\nabla^\mu$
denote its Levi-Civita connection. Let $\nabla^\mathfrak{g}$ denote
the connection which makes left-invariant vectorfields parallel. Since
$\nabla^\mu-\nabla^\mathfrak{g}$ is left-invariant, hence
$\nabla^\mathfrak{g}$-parallel, higher covariant derivatives computed
using either $\nabla^\mathfrak{g}$ or $\nabla^\mu$ determine each
other via bounded expressions. Therefore the Riemannian Sobolev space
$W^{h,p}$ can be alternatively defined using $\nabla^\mathfrak{g}$,
i.e. using derivatives along left-invariant vectorfields. Since every
left-invariant vectorfield is a combination of compositions of at most
$r$ horizontal derivatives $X_i$, where $r$ is the step of
$\mathfrak{g}$,
\begin{eqnarray}\label{eq:8}
  W^{h,p}\subset W_c^{h,p}\subset W^{h/r,p}.
\end{eqnarray} 

According to Proposition~\ref{prop:leray}, it is sufficient to find
some large integer $h$ such that every (usual) closed differential
form $\alpha \in W^{h,p}\Omega^k(G)$ writes $\alpha = d \beta$ with
$\beta \in W^{H,q}\Omega^{k-1}(G)$ for $H=n^{n+1}+1$. The relevant
value of $h$ will arise from the proof.

We first retract $\alpha$ in the sub-complex $(E,d)$ using the
differential homotopy $\Pi_E$. Indeed from
Theorem~\ref{thm:dc-complex},
\begin{displaymath}
  \alpha = \Pi_E \alpha + d R\alpha + R d \alpha = \Pi_E \alpha + d R
  \alpha \,,
\end{displaymath}
where $\Pi_E$ and $R$ are left invariant differential operators of
(horizontal) order at most $Q$, the homogeneous dimension of
$\mathfrak{g}$.  Then $\Pi_E \alpha$, $R\alpha$ and $dR \alpha$ all
belong to $W^{h-Q,p}_c$, so that $\alpha$ and $\Pi_E \alpha$ are
homotopic in $L^{p,p}$, and \emph{a fortiori} $L^{q,p}$ Sobolev
cohomology for $q\geq p$ (see Remark \ref{rem:decreasing}).

We now deal with $\alpha_E= \Pi_E \alpha$. Its algebraic projection on
$E_0$, $\alpha_c = \Pi_{E_0} \alpha_E$ is a contracted $d_c$-closed
differential form in $W_c^{h-Q,p}$ too. Since
$\alpha_c \in E_0^k \simeq H^k(\mathfrak{g})$, the weights of its
components belong to the interval $[a,b]$ where $a=w_{min} (k)$ and
$b=w_{max} (k)$. Moreover, since $K_c \alpha_c \in E_0^{k-1}$, it
belongs to $\Omega_{[a',b']}$ with $a' = w_{min}(k-1)$ and
$b' = w_{max}(k-1)$.

We can now apply Corollary~\ref{cor:K_c_Sobolev}, with
$$
\mu = b- a' = w_{max}(k) - w_{min}(k-1) = \delta N_{max}(k)
$$ 
and $h' = h-Q + a-b + a' -b'$. (Observe that $\mu < Q$ except for
$G=\R$ in which case $L^1$ one forms have bounded primitives.) We get
that $\beta_c = K_c \alpha_c \in W^{h',q}_c$ with
$d_c \beta_c = \alpha_c$ and
$\frac{1}{p}-\frac{1}{q}=\frac{\delta N_{max}(k)}{Q}$.

Finally, let $ \beta_E=\Pi_E \beta_c$. Then $\beta_E\in W_c^{h'-Q,q}$
since $\Pi_E$ is a differential operator of order at most $Q$. By
construction, $\alpha_E=d\beta_E$. By inclusions \eqref{eq:8},
$\beta_E$ belongs to the Riemannian Sobolev space
$W^{(h'-Q)/r,q}$. Thus let us choose $h=rH + 2Q+b-a+b'-a'$. We have
shown that every $k$-form in $W^{h,p}$ is homotopic in $L^{p,p}$
Sobolev cohomology to a form that has a primitive in $W^{H,q}$, where
\begin{displaymath}
  \frac{1}{p}-\frac{1}{q}=\frac{\mu}{Q}= \frac{\delta
    N_{max}(k)}{Q}. 
\end{displaymath}
Proposition~\ref{prop:leray} implies that
$\ell^{q,p}H^k(G)=0$. \emph{A fortiori}, $\ell^{q',p}H^k(G)=0$ for all
$q'\geq q$.

\subsection{Remarks on the weight gaps}
\label{sec:remarks-weight-gaps}

Implicit in the statement of Theorem~\ref{1} is that for any
$1\leq k \leq n $, the weight gap
$\delta N(k) = w_{max}(k)- w_{min}(k-1)$ is positive. Actually, one
has
\begin{equation}
  \label{eq:9}
  w_{max}(k) - w_{max}(k-1) \geq 1 \quad  \mathrm{and}\quad  w_{min}(k) -
  w_{min}(k-1)  \geq 1 .
\end{equation}

\begin{proof}
  Since $E_0 = \ker d_0 \cap \ker \delta_0$ is Hodge-$*$ symmetric
  with $w(*\alpha) = Q - w (\alpha)$ (see e.g.
  \cite{Rumin_CRAS_1999}), one can restrain to prove the statement
  about $w_{min}$. Let $\alpha \in H^{k}(\mathfrak{g})$ be non zero
  with minimal weight $w_{min}(k)$. See it as a left invariant
  retracted form in $E_0^k$. Following Section~\ref{thm:dc-complex},
  one has then $d \alpha = d_0 \alpha = d_c \alpha = 0$. Now
  $(E_0,d_c)$ being locally exact (homotopic to de Rham complex), one
  has $\alpha = d_c \beta$ for some $\beta \in E_0^{k-1}$. Since $d_c$
  increases the weight by $1$ at least (as $d_0=0$ on it), $\beta$ has
  a non vanishing component of weight $< w_{min}(k)$, whence
  $w_{min}(k-1) \leq w_{min}(k) -1$.
  
\end{proof}

Another remark is about the use of the contracted complex here. As
observed in Remark~\ref{rem:1}, de Rham complex being C-C elliptic
too, one could directly use a similar homotopy $K$ for it in the
previous proof. But then, it would lead to a weaker integration result
of closed $W_c^p$ forms in $W_c^q$ for a larger gap
$\frac{1}{p} - \frac{1}{q} = \frac{\delta N}{Q}$. Indeed this
$\delta N $ is the maximal weight gap between the whole $\Omega^k(G)$
and $\Omega^{k-1}(G)$, instead of the restricted ones. This gives a
weaker vanishing condition of $\ell^{q,p} H^{k}(G)$. The point here is
that the first homotopy $\Pi_E$ to $E$, being differential, ``costs a
lot of derivatives'' that would be a shame locally, but is harmless
here when working with the quite smooth $W_c^{h,p}$ forms coming from
our $\ell^p$ simplicial cocycles. Indeed, we have seen that
$\Pi_E = 1$ in $L^{p,p}$ Sobolev cohomology.

\smallskip

Still about these ideas of ``loosing'' or ``gaining'' derivatives,
things go exactly in the opposite direction as usual here. Namely, one
sees in the proof that the more derivatives the homotopy $K_c$
controls at some place, the larger becomes the $(p,q)$ gap from
Hardy-Littlewood-Sobolev inequality in
Corollary~\ref{cor:K_c_Sobolev}. This is of course better in local
problems since $L^q_{loc}$ gets smaller, but is weaker on global
smooth forms as $W_c^q$ gets larger. In this large scale low frequency
integration problem, the less derivatives you gain is the better. No
gain, no pain.

\section{Nonvanishing result}
\label{nonvanishing}

As we will see, in order to prove the non-vanishing of the
$\ell^{q,p}$ and Sobolev $L^{q,p}$ cohomology of $G$ for some
$1\leq p \leq q\leq +\infty$, we shall construct closed $k$-forms
$\omega\in W^{h,p}(G)$, $h$ large enough, and $n-k$-forms $\omega'_j$
such that $\|d\omega'_j\|_{q'}$ tend to zero whereas
$\int_{G}\omega\wedge\omega'_j$ stays bounded away from $0$.  Here
$q'$ is the dual exponent of $q$ : $\frac{1}{q}+\frac{1}{q'}=1$.

The building blocks will be differential forms which are homogeneous
under dilations $\delta_t$. In this section, one can use any expanding
one-parameter group $s\mapsto h_s$ of automorphisms of $G$. Expanding
means that the derivation $D$ generating $h_s$ has positive
eigenvalues.  The one-parameter group $s\mapsto \delta_{e^s}$ is an
example, but others may be useful, see below.

\subsection{Homogeneous differential forms}
\label{sec:homog-diff-forms}

Fix an expanding one-parameter group $s\mapsto h_s$ of automorphisms
of $G$, generated by a derivation $D$. The data $(G,(h_s))$ is called
a \emph{homogeneous Lie group}. Denote by $T=trace(D)$ its
\emph{homogeneous dimension}.

Left-invariant differential forms on $G$ split into weights $w$ under
$h_s$.

Say a smooth differential form $\omega$ on $G\setminus\{1\}$ is
homogeneous of degree $\lambda$ if $h_s^*\omega=e^{s\lambda} \omega$
for all $s\in\R$.  Note that homogeneity is preserved by $d$. A
left-invariant differential form of weight $w$ is homogeneous of
degree $w$

Let $\rho$ denote a continuous function on $G$ which is homogeneous of
degree $1$, and smooth and positive away from the origin. Let $\beta$
be a nonzero continuous differential form which is homogeneous of
degree $\lambda$, smooth away from the origin and has weight $w$, then
\begin{eqnarray*}
  \beta=\rho^{\lambda-w}\sum a_i\theta_i ,
\end{eqnarray*}
where $\theta_i$'s are left-invariant of weight $w$ and $a_i$ are
smooth homogeneous functions of degree $0$, hence are
bounded. Therefore
\begin{displaymath}
  \beta\in L^p(\{\rho\geq 1\}) \iff
  \int_{1}^{+\infty}\rho^{p(\lambda-w)}\rho^{T-1}\, d\rho<\infty
  \iff \lambda-w +\frac{T}{p}<0.
\end{displaymath}
It follows that if $\gamma$ is a differential form which is
homogeneous of degree $\lambda$ and has weight $\geq w$,
\begin{equation}
  \label{eq:10}
  \lambda-w +\frac{T}{p}<0 \implies \gamma\in L^p(\{\rho\geq 1\}).
\end{equation}

Start with a closed differential $k$-form $\omega$ which is
homogeneous of degree $\lambda$ and of weight $\geq w$ (and no
better). Pick a differential $n-k$-form $\alpha'$ which is homogeneous
of degree $\lambda'$. Assume that $d \alpha'$ has weight $\geq w'$
(and no better). Set
\begin{eqnarray*}
  \omega'_j=\chi_j \alpha',
\end{eqnarray*}
where $\chi_j=\chi\circ\rho\circ h_{-j}$ and $\chi$ is a cut-off
supported on $[1,2]$. The top degree form $\omega\wedge \alpha'$ is
homogeneous of degree $\lambda+\lambda'$ and has weight $T$. It
belongs to $L^1$ if $\lambda+\lambda'<0$.  Thus, in order that
$\int\omega\wedge \omega'_j$ does not tend to $0$, it is necessary
that $\lambda+\lambda'\geq 0$.

By construction,
\begin{eqnarray*}
  \omega'_j=e^{\lambda'j} h_{-j}^*\omega'_1.
\end{eqnarray*}
Let $\beta$ denote a component of $d\omega'_1$ of weight $\tilde w$,
and $s=-j$. By the change of variable formula,
\begin{align*}
  \| h_s^*\beta\|_{q'}^{q'}
  &= \int | h_s^*\beta|^{q'}\,dvol\\
  &= \int e^{q'{\tilde w}s}|\beta|^{q'}\circ h_s\,dvol\\
  &=  e^{(q'{\tilde w}-T)s}\int |\beta|^{q'}\,dvol\\
  &= e^{(q'{\tilde w}-T)s}\|\beta\|_{q'}^{q'}.
\end{align*}
This works as well for $q'=\infty$. Since $d\omega'_1$ has weight
$\geq w'$,
\begin{eqnarray*}
  \|\omega'_j\|_{q'}\leq e^{(\lambda'-w'+\frac{T}{q'})j}\|\omega'_1\|_{q'}.
\end{eqnarray*}
One concludes that
\begin{eqnarray*}
  \|\omega'_j\|_{q'}\to 0 \iff \lambda'-w'+\frac{T}{q'}<0.
\end{eqnarray*}
It holds for $q'=\infty$ as well. Remember that
$\omega\in L^p \iff \lambda-w+\frac{T}{p}<0$. Note that
\begin{eqnarray*}
  \lambda-w+\frac{T}{p}+\lambda'-w'+\frac{T}{q'}
  =\lambda+\lambda'-w-w'+T+T(\frac{1}{p}-\frac{1}{q}).
\end{eqnarray*}
Finally, one sees that one can pick $\lambda$ and $\lambda'$ such that
$\omega\in L^p$, $\|\omega'_j\|_{q'}\to 0$ and
$\lambda+\lambda'\geq 0$ iff
\begin{equation}
  \label{eq:11}
  \frac{1}{p}-\frac{1}{q}<\frac{w+w'-T}{T}.
\end{equation}

\subsection{A numerical invariant of homogeneous groups}

A lower bound for the sum $w+w'$ appearing in above inequation
\eqref{eq:11} is provided by the following definitions.

Let $\Sigma$ denote the level set $\{\rho=1\}$. It is a smooth compact
hypersurface, transverse to the vectorfield $\xi$ which generates the
1-parameter group $s\mapsto h_{e^s}$. Differential $k$-forms which are
homogeneous of degree $\lambda$ on $G\setminus\{1\}$ correspond to
smooth sections of the pull-back of the bundle $\Lambda^k T^* G$ by
the injection $\Sigma\hookrightarrow G$. Given such a section
$\sigma$, a form $\alpha$ is defined as follows. At a point $x$ where
$\rho(x)=e^s$, $\alpha(x)= h_{-s}^*\sigma$. Conversely, given a
homogeneous form $\alpha$, consider its values $\sigma$ along $\Sigma$
(not to be confused with the restriction of $\alpha$, which belongs to
$\Lambda^k T^* \Sigma$). A similar construction applies to contracted
forms as $E_0$ is stable by dilations.

On spaces of homogeneous forms of complementary degrees $k$ and $n-k$
and complementary degrees of homogeneity $\lambda$ and $-\lambda$,
define a pairing as follows: if $\beta$ and $\beta'$ are homogeneous
of degrees $\lambda$ and $-\lambda$, set
$$
I(\beta,\beta')=\int_{\Sigma}\iota_\xi(\beta\wedge\beta').
$$
This is a nondegenerate pairing. Indeed, pointwise, an $n$-form
$\omega$ is determined by the restriction of $\iota_\xi(\omega)$ to
$T\Sigma$, hence the pointwise pairing
$(\beta,\beta')\mapsto \iota_\xi(\beta\wedge\beta')_{|T\Sigma}$ is
nondegenerate. For instance, one has
$\beta \wedge * \beta = \|\beta\|^2 \,d vol$ pointwise. Note that the
$n$-form $\beta\wedge\beta'$ is homogeneous of degree $0$,
i.e. dilation invariant. The $n-1$-form $\iota_\xi(\beta\wedge\beta')$
is closed, the integral $I(\beta,\beta')$ only depends on the
cohomology class of this form. The boundary of any smooth bounded
domain containing the origin can be used to perform integration
instead of $\Sigma$.


\begin{dfi}
  \label{dfi:wsG}
  Let $G$ be a homogeneous Lie group of homogeneous dimension $T$. For
  $k=1,\ldots,n= \mathrm{dim}(G)$, define $ws_G(k)$ as the maximum of
  sums $w+w'-T$ such that for a dense set of real numbers $\lambda$,
  there exist
  \begin{enumerate}
  \item a differential form $\alpha$ of degree $k-1$ on
    $G\setminus\{1\}$, homogeneous of degree $\lambda$, such that
    $d \alpha$ has weight $\geq w$,
  \item a differential form $\alpha'$ of degree $n-k$ on
    $G\setminus\{1\}$, homogeneous of degree $\lambda'=-\lambda$, such
    that $d \alpha'$ has weight $\geq w'$,
  \end{enumerate}
  such that $ I(d \alpha,\alpha')\not=0$.
\end{dfi}
Note that for all $k\geq 1$, $ws_G(k)=ws_G(n-k+1)$. For instance, when
Carnot dilations are used, nonzero $1$-forms of weight $\geq 2$ are
never closed, and $n$-forms are always closed and of weight $T$, hence
$ws_G(1)=ws_G(n)=1$.

\subsection{Cohomology nonvanishing}

\begin{thm}
  \label{thm:nonzero}
  Let $G$ be a homogeneous Lie group of homogeneous dimension
  $T$. Then $\ell^{q,p}H^k(G)\not=0$ provided
  \begin{eqnarray*}
    1\leq p,q<+\infty \quad \text{and} \quad
    \frac{1}{p}-\frac{1}{q}<\frac{ws_G(k)}{T}. 
  \end{eqnarray*}
\end{thm}

\begin{proof}
  By assumption,
  $\epsilon=\frac{w+w'-T}{T}-\frac{1}{p}+\frac{1}{q}>0$. Pick a real
  number $\lambda$ in the dense set given in Definition~\ref{dfi:wsG}
  and such that
  \begin{equation}
    \label{eq:12}
    w-\frac{T}{p}-\frac{T}{2}\epsilon \leq \lambda < w-\frac{T}{p} .
  \end{equation}
  Then $\lambda'=-\lambda$ satisfies
  \begin{align}
    \lambda'-w'+\frac{T}{q'}
    & = -\lambda - w' + \frac{T}{q'} \nonumber \\
    & \leq  -w + \frac{T}{p} + \frac{T\epsilon}{2}
      -w' +   \frac{T}{q'}
      = - \frac{T\epsilon}{2} < 0
      . \label{eq:13}
  \end{align}
  By definition, there exist differential $k-1$ and $n-k$-forms
  $\alpha$ and $\alpha'$, homogeneous of degrees $\lambda$ and
  $\lambda'$, such that $d \alpha$ and $d\alpha'$ have weights
  $\geq w$ and $\geq w'$. Then $d\alpha \wedge \alpha'$ is homogeneous
  of degree $0$. Using the notations of
  Section~\ref{sec:homog-diff-forms}, for all $j$,
  $d \alpha\wedge \chi_j\alpha'= h_{-j}^*(d \alpha\wedge \chi_1
  \alpha')$, hence
  $$
  \int_{G} d\alpha \wedge \chi_j\alpha'=\int_{G} d\alpha \wedge
  \chi_1\alpha'=I(d\alpha ,\alpha')\int_{\R}\chi(t)\,dt\not=0.
  $$
  Since $d \alpha$ is homogeneous of degree $\lambda$ and has weight
  $\geq w$, it belongs to $L^p$ (away from a neighborhood of the
  origin) by \eqref{eq:10} and \eqref{eq:12}. Furthermore, derivatives
  along left invariant vector fields decrease homogeneity. Hence all
  such derivatives of $d\alpha$ belong to $L^p$. After smoothing
  $\alpha$ near the origin, we get a closed form $\omega$ on $G$ that
  coincides with $d\alpha$ on $\{\rho \geq 1\}$, and which belongs to
  $W^{h,p}$ for all $h$. Set
  \begin{eqnarray*}
    \omega'_j=\chi_j \alpha'.
  \end{eqnarray*}
  Then $\int_{G}\omega\wedge \omega'_j$ does not depend on $j$.

  Assume by contradiction that $\omega=d\phi$ where $\phi\in W^{h,q}$.
  In particular, $\phi\in L^q$. Since $\omega'_j$ are compactly
  supported, Stokes theorem applies and
  \begin{eqnarray*}
    |\int_{G}\omega\wedge \omega'_j|
    &=&|\int_{G}d\phi\wedge \omega'_j|\\
    &=&|\int_{G}\phi\wedge d\omega'_j|\\
    &\leq&\|\phi\|_{q}\|d\omega'_j\|_{q'}
  \end{eqnarray*}
  which tends to $0$, as $\alpha'$ and
  $d \alpha' \in L^q(\{\rho \geq 1\}$) by \eqref{eq:10} and
  \eqref{eq:13}, contradiction. We conclude that $[\omega]\not=0$ in
  the Sobolev $L^{q,p}$ cohomology of $G$.  According to Proposition
  \ref{prop:leray}, this implies that the $\ell^{q,p}$ cohomology of
  $G$ does not vanish.
\end{proof}

\subsection{Lower bounds on $ws_G$}

We give here two lower bounds on $ws_G$. Combined with Theorem
\ref{thm:nonzero}, they complete the proof of Theorem \ref{1}(ii) in
the wider setting of homogeneous groups. We start with a lemma on the
contracted complex.

\begin{lem}
  \label{lem:existence_formes}
  Let $G$ be a homogeneous Lie group of dimension $n$, let
  $k=1,\ldots,n$.  Then for an open dense set of real numbers
  $\lambda$, there exist smooth non $d_c$-closed contracted
  $k-1$-forms on $G \setminus \{1\}$ which are homogeneous of degree
  $\lambda$.
\end{lem}

\begin{proof}
  
  The differential $d_c \not\equiv 0$ on $E_0^{k-1}$, since the
  complex $(d_c, E_0^*)$ is a resolution on $G$.  Then, by the
  Stone-Weierstrass approximation theorem, their exist non
  $d_c$-closed contracted forms with homogeneous polynomial components
  in an invariant basis. Pick one term $P \alpha_0$ with
  $\alpha_0\in E_0^{k-1}$ invariant, and a non constant homogeneous
  polynomial $P$ such that $d_c (P\alpha_0)\not=0$. Up to changing $P$
  into $-P$, pick $x_0 \in G$ such that $d_c(P\alpha_0) (x_0) \not=0$
  and $P(x_0) >0$. Consider the map
  \begin{displaymath}
    F \ :\ \lambda \in \C \mapsto d_c (P^\lambda \alpha_0) (x_0) . 
  \end{displaymath}
  Since $d_c$ is a differential operator, $F$ is analytic. Since
  $F(1) \not=0$, one has $F(\lambda) \not= 0$ except for a set of
  isolated values of $\lambda$. Let $\chi$ be a smooth homogeneous
  function on $G \setminus \{1\}$ of degree $0$ with support in
  $\{P> 0\}$ and $\chi=1$ around $x_0$. Then
  $\alpha = \chi P^\lambda \alpha_0$ is a smooth non $d_c$-closed
  homogeneous contracted form on $G\setminus \{1\}$ of degree
  $w(\alpha)= \lambda w(P) + w(\alpha_0)$.

\end{proof}

\begin{pro}
  \label{lower}
  Let $G$ be a homogeneous Lie group of dimension $n$. For all
  $k=1,\ldots,n$,
  \begin{displaymath}
    ws_G(k)\geq \max\{1, w_{min}(k)-w_{max}(k-1)\}.
  \end{displaymath}
\end{pro}

\begin{proof}

  By Lemma~\ref{lem:existence_formes}, pick a non $d_c$-closed
  contracted $k-1$-form $\alpha$, homogeneous of degree
  $\lambda$. Assume that $d_c\alpha$ has weight $\geq w$ and no better
  (i.e. its weight $w$ component $(d_c \alpha)_w$ does not vanish
  identically).  Pick a smooth contracted $n-k$-form $\alpha'$ of
  weight $T-w$, homogeneous of degree $-\lambda$ and such that
  $I(d_c\alpha,\alpha')\not=0$.  For instance
  $\alpha'= \rho^{-2\lambda+ 2w -T} * (d_c\alpha)_w$ will do. Set
  $\alpha_E = \Pi_E \alpha$ and $\alpha_E'= \Pi_E \alpha'$.

  By construction (see Theorem~\ref{thm:dc-complex}),
  $\Pi_E = \Pi_{E_0}+ D$ where $D$ strictly increases the
  weight. Hence $d\alpha_E - d_c \alpha= \Pi_E d_c \alpha - d_c\alpha$
  has weight $\geq w +1$, and $\alpha'_E - \alpha'$ has weight
  $\geq T-w + 1$. Therefore
  $d\alpha_E \wedge \alpha'_E - d_c \alpha \wedge \alpha'$ has weight
  $\geq T+1$, thus vanishes. Then it holds that
  \begin{displaymath}
    I(d\alpha_E, \alpha'_E) = I(d_c \alpha, \alpha') \not=0 \,.
  \end{displaymath}
  Consider now the weight of $d\alpha'_E$. By construction,
  $E \subset \ker d_0$, so that $d$ strictly increases the weight on
  $E$, see Section~\ref{sec:dc_complex}. Therefore
  $$
  w(d\alpha'_E) \geq w(\alpha'_E) + 1 = w(\alpha')+ 1 = T - w(d_c
  \alpha) + 1 = T -w(d \alpha_E) + 1,
  $$
  hence $ws_G(k) \geq w(d \alpha_E) + w (d \alpha'_E) - T \geq 1 $ as
  needed. One has also that
  $$
  w(d\alpha'_E) = w (d_c \alpha') \geq w_{min}(n-k+1) = T -
  w_{max}(k-1),
  $$
  by Hodge $*$-duality, see proof of \eqref{eq:9}, while
  $$
  w(d\alpha_E) = w (d_c \alpha) \geq w_{min}(k) \,.
  $$
  This gives
  $ws_G(k) \geq w(d \alpha_E) + w (d \alpha'_E) - T \geq w_{min}(k) -
  w_{max}(k-1)$.

\end{proof}


\subsection{An example: Engel's group}
\label{sec:engels-group}

We illustrate the non-vanishing results on the Engel group $E^4$.

It has a 4-dimensional Lie algebra with basis $X,Y,Z,T$ and nonzero
brackets $[X,Y]=Z$ and $[X,Z]=T$. One finds, see
e.g. \cite[Section~2.3]{Rumin_TSG}, that
$$
H^1(\mathfrak{g}) \simeq \mathrm{span}(\theta_X, \theta_Y) \
\mathrm{and} \ H^2(\mathfrak{g}) \simeq \mathrm{span}(\theta_X \wedge
\theta_Z, \theta_Y \wedge \theta_Z) \,.
$$
The following table gives the values of $\delta N_{max}$ and
$\delta N_{min}$ for $E^4$ with respect to its standard Carnot weight
: $w(X)= w(Y) =1$, $w(Z)= 2$ and $w(T) = 3$. One has $Q=7$ and

\begin{center}
  \begin{tabular}{|l|c|c|c|c|}
    \hline
    k   & 1 & 2 & 3 & 4 \\
    \hline\hline
    $w_{max}(k)$ & 1 & 4 & 6 & 7 \\
    \hline
    $w_{min}(k)$ & 1 & 3 & 6 & 7 \\
    \hline
    $\delta N_{max}(k)$   & 1 & 3 & 3 & 1 \\
    \hline
    $\delta N_{min}(k)$   & 1 & 2 & 2 & 1 \\
    \hline
  \end{tabular}
\end{center}

We see that Theorem \ref{1} is sharp in degrees $1$ and $4$. However,
there are gaps in degrees $2$ and $3$. In particular, $H^{2,q,p}(E^4)$
vanishes when $\frac{1}{p}- \frac{1}{q} \geq \frac{3}{7}$ and does not
when $\frac{1}{p}- \frac{1}{q} < \frac{2}{7}$, provided
$1<p, q< +\infty$.

\smallskip

Following \cite[Section~4.2]{Rumin_TSG}, let us also use the expanding
one-parameter group of automorphisms of $E^4$ generated by the
derivation $D$ defined by
\begin{eqnarray*}
  D(X)=X ,\quad D(Y)=2Y ,\quad D(Z)=3Z ,\quad D(T)= 4T.
\end{eqnarray*}
Then $trace(D)=10$, and with this choice of derivation, the table of
weights becomes

\begin{center}
  \begin{tabular}{|l|c|c|c|c|}
    \hline
    k   & 1 & 2 & 3 & 4 \\
    \hline\hline
    $w_{max}(k)$ & 2 & 5 & 9 & 10 \\
    \hline
    $w_{min}(k)$ & 1 & 5 & 8 & 10 \\
    \hline
    $\delta N_{max}(k)$   & 2 & 4 & 4 & 2 \\
    \hline
    $\delta N_{min}(k)$   & 1 & 3 & 3 & 1 \\
    \hline
  \end{tabular}
\end{center}

According to Proposition \ref{lower}, with respect to this homogeneous
structure, $ws_{E^4}(2)\geq \delta N_{min}(2)=3$. Then with Theorem
\ref{thm:nonzero},
\begin{eqnarray*}
  1\leq p,q<+\infty \quad \text{and} \quad \frac{1}{p}-\frac{1}{q}<\frac{3}{10}
\end{eqnarray*}
implies that $\ell^{q,p}H^2(E^4)\not= \{0\}$. We see that a non-Carnot
homogeneous structure may yield a better interval for cohomology
nonvanishing, which is intriguing for a large scale geometric
invariant.

\bibliographystyle{abbrv}



\begin{thebibliography}{}

\end{thebibliography}


\begin{thebibliography}{10}

\bibitem{B-MinOo-Ruh}
J.~Bemelmans, Min-Oo, and E.~A. Ruh.
\newblock Smoothing {R}iemannian metrics.
\newblock {\em Math. Z.}, 188(1):69--74, 1984.

\bibitem{Bourdon-Kleiner}
M.~Bourdon and B.~Kleiner.
\newblock Some applications of {$\ell_p$}-cohomology to boundaries of {G}romov
  hyperbolic spaces.
\newblock {\em Groups Geom. Dyn.}, 9(2):435--478, 2015.

\bibitem{Bourdon-Pajot}
M.~Bourdon and H.~Pajot.
\newblock Cohomologie {$l_p$} et espaces de {B}esov.
\newblock {\em J. Reine Angew. Math.}, 558:85--108, 2003.

\bibitem{CGGP}
M.~Christ, D.~Geller, P.~G{\l}owacki, and L.~Polin.
\newblock Pseudodifferential operators on groups with dilations.
\newblock {\em Duke Math. J.}, 68(1):31--65, 1992.

\bibitem{Cornulier-Tessera}
Y.~Cornulier and R.~Tessera.
\newblock Contracting automorphisms and {$L^p$}-cohomology in degree one.
\newblock {\em Ark. Mat.}, 49(2):295--324, 2011.

\bibitem{Drutu-Mackay}
C.~Drutu and J.~M. Mackay.
\newblock Random groups, random graphs and eigenvalues of p-laplacians.
\newblock arXiv:1607.04130, 2016.

\bibitem{Elek}
G.~Elek.
\newblock Coarse cohomology and {$l_p$}-cohomology.
\newblock {\em $K$-Theory}, 13(1):1--22, 1998.

\bibitem{Folland_1975}
G.~B. Folland.
\newblock Subelliptic estimates and function spaces on nilpotent {L}ie groups.
\newblock {\em Ark. Mat.}, 13(2):161--207, 1975.

\bibitem{Genton}
L.~Genton.
\newblock {\em Scaled Alexander-Spanier Cohomology and $L^{q,p}$ Cohomology for
  Metric Spaces}.
\newblock PhD thesis, EPFL, Lausanne, Th\`ese $n^0$ 6330, 2014.

\bibitem{Gromov}
M.~Gromov.
\newblock Asymptotic invariants of infinite groups.
\newblock In {\em Geometric group theory, {V}ol.\ 2 ({S}ussex, 1991)}, volume
  182 of {\em London Math. Soc. Lecture Note Ser.}, pages 1--295. Cambridge
  Univ. Press, Cambridge, 1993.

\bibitem{Pansu2}
P.~Pansu.
\newblock Cup-products in $l^{q,p}$-cohomology: discretization and
  quasi-isometry invariance.
\newblock arXiv 1702.04984, 2017.

\bibitem{Rumin_CRAS_1999}
M.~Rumin.
\newblock Differential geometry on {C}-{C} spaces and application to the
  {N}ovikov-{S}hubin numbers of nilpotent {L}ie groups.
\newblock {\em C. R. Acad. Sci. Paris S\'er. I Math.}, 329(11):985--990, 1999.

\bibitem{Rumin_TSG}
M.~Rumin.
\newblock Around heat decay on forms and relations of nilpotent {L}ie groups.
\newblock In {\em S\'eminaire de {T}h\'eorie {S}pectrale et {G}\'eom\'etrie,
  {V}ol. 19, {A}nn\'ee 2000--2001}, volume~19 of {\em S\'emin. Th\'eor. Spectr.
  G\'eom.}, pages 123--164. Univ. Grenoble I, Institut Fourier,
  Saint-Martin-d'H\`eres, 2001.

\end{thebibliography}

\def\cprime{$'$}

-----------------------------------------------------

\end{document}